\documentclass[11pt]{amsart}
\usepackage{latexsym,amscd,amssymb, graphicx, color, amsthm, bm, amsmath, cancel, enumitem}  
\usepackage{tikz}
\usepackage{tikz-cd}
\usepackage[margin=1in]{geometry}
\usepackage{hyperref}
\usepackage[all,cmtip]{xy}
\usepackage{ytableau}

\numberwithin{equation}{section}

\newtheorem{theorem}{Theorem}[section]
\newtheorem{proposition}[theorem]{Proposition}
\newtheorem{corollary}[theorem]{Corollary}
\newtheorem{lemma}[theorem]{Lemma}

\newtheorem{question}[theorem]{Question}

\newtheorem{example}[theorem]{Example}
\newtheorem{remark}[theorem]{Remark}
\newtheorem{definition}[theorem]{Definition}

\theoremstyle{definition}

\newcommand{\Hilb}{{\mathrm{Hilb}}}

\newcommand{\symm}{{\mathfrak{S}}}

\newcommand{\gr}{{\mathrm {gr}}}
\newcommand{\Gr}{{\mathrm {Gr}}}
\newcommand{\BBB}{{\mathcal{B}}}

\newcommand{\AAA}{{\mathcal{A}}}

\newcommand{\CCC}{{\mathcal{C}}}

\newcommand{\CC}{{\mathbb{C}}}
\newcommand{\QQ}{{\mathbb{Q}}}
\newcommand{\ZZ}{{\mathbb{Z}}}
\newcommand{\PP}{{\mathbb{P}}}

\newcommand{\xxx}{{\mathbf{x}}}

\newcommand{\ttt}{{\mathbf{t}}}
\newcommand{\uuu}{{\mathbf{u}}}

\newcommand{\sss}{{\mathbf{s}}}

\newcommand{\aalpha}{{{\bm \alpha}}}

\newcommand{\im}{\text{im}}

\newcommand{\edge}{\text{edge}}

\newcommand{\Fl}{\text{Fl}}
\DeclareMathOperator{\Spec}{Spec}

\begin{document}

\title[Torus equivariant cohomology for the $\Delta$-Springer Fiber]
{Torus Equivariant Cohomology for the $\Delta$-Springer Fiber}

\author[Raymond Chou]{Raymond Chou}
\email{r2chou@ucsd.edu}

\begin{abstract}
We define a torus $U \subset T = (\CC^\times)^K$ which acts on the $\Delta$-Springer varieties $Y_{n,\lambda,s}$ defined by Griffin-Levinson-Woo \cite{GLW} and give a Borel-style presentation for the equivariant cohomology ring $H^*_U(Y_{n,\lambda,s})$. Our presentation arises from the orbit harmonics deformation technique, and uses methods of Chou-Matsumura-Rhoades \cite{CMR} and Abe-Horiguchi \cite{AH}.
\end{abstract}

\maketitle

\section{Introduction}\label{sec:introduction}

A central theme in Schubert calculus is to describe cohomology rings of naturally
occurring varieties by explicit generators and relations \cite{AH,AF,CMR,HaradaHoriguchiMasuda2015}. Borel's presentation of
the cohomology ring of the complete flag variety $\Fl_n(\CC^n)$ as the coinvariant algebra $R_n$ is the foundational example \cite{Borel1953}, and analogous presentations for Springer fibers and their
generalizations have played a central role in geometric representation theory and
algebraic combinatorics. Equivariant cohomology enhances these rings by accounting for the action of a torus, and equivariant Borel-style presentations often reveal structure
that is invisible in ordinary cohomology.

The varieties studied in this paper are the \emph{$\Delta$-Springer fibers}
$Y_{n,\lambda,s}$ introduced by Griffin-Levinson-Woo \cite{GLW}. These are
subvarieties of partial flag varieties depending on a partition
$\lambda \vdash k \leq n$ and an integer $s \geq \ell(\lambda)$. Their ordinary
cohomology is described in \cite{GLW} in terms of the $\Delta$-Springer modules
$R_{n,\lambda,s}$ introduced by Griffin \cite{G}, which simultaneously generalize
the Garsia-Procesi modules \cite{GarsiaProcesi1992} and the generalized coinvariant rings $R_{n,k}$ of Haglund-Rhoades-Shimozono \cite{HRS}. Thus the geometry of $Y_{n,\lambda,s}$ lies at the intersection of Springer theory and the algebraic combinatorics surrounding Delta-type coinvariant constructions.

In the classical Springer fiber setting, Kumar-Procesi \cite{KumarProcesi2012} gave a coordinate ring presentation of the torus-equivariant cohomology of the Springer fiber. Abe-Horiguchi \cite{AH} upgraded this to a Borel presentation with explicit generators and relations by extending the argument of Tanisaki \cite{T} for ordinary cohomology. The generalized coinvariant rings $R_{n,k}$ were realized as the cohomology ring of the moduli space of \emph{spanning line configurations} defined by Pawlowski-Rhoades \cite{PawlowskiRhoades2019}, whose equivariant cohomology was computed by Chou-Matsumura-Rhoades \cite{CMR} as a Borel-style presentation. The varieties $X_{n,k}$ are smooth but not compact; Griffin-Levinson-Woo \cite{GLW} give an alternate geometric model as a special case of the varieties $Y_{n,\lambda,s}$, which are compact yet singular.

The goal of this paper is to establish the analogous Borel presentation for the torus-equivariant cohomology for the full
$\Delta$-Springer family. To do so, we define a natural subtorus
\[
U \subset T = (\CC^\times)^K
\]
which preserves $Y_{n,\lambda,s}$, and we compute the equivariant cohomology ring
$H_U^*(Y_{n,\lambda,s})$ explicitly.

Our presentation takes the form of a quotient of the polynomial ring
$\QQ[\xxx_n,\uuu_s]$. The defining ideal is generated by two families of
relations. The first family,
\[
(x_i-u_1)\cdots(x_i-u_s) \qquad (1 \leq i \leq n),
\]
arises from the iterated projective bundle structure of $Y_{n,\emptyset,s}$.
The second family is given by relations we call \emph{double Tanisaki polynomials} in subsets of the
variables $x_1,\dots,x_n$, and is obtained from equivariant Schubert classes on
Grassmannians. Writing $J_{n,\lambda,s}$ for the ideal generated by these two
families, our main result identifies
\[
H_U^*(Y_{n,\lambda,s}) \cong \QQ[\xxx_n,\uuu_s]/J_{n,\lambda,s}.
\]
After specializing $u_1=\cdots=u_s=0$, this presentation recovers the
Griffin-Levinson-Woo presentation of the ordinary cohomology ring
$H^*(Y_{n,\lambda,s})$. In this sense, the quotient
$\QQ[\xxx_n,\uuu_s]/J_{n,\lambda,s}$ may be viewed as an equivariant deformation
of the $\Delta$-Springer module $R_{n,\lambda,s}$.

A second theme of the paper is that the shape of this equivariant presentation is
predicted by the phenomenon known as \emph{orbit harmonics}. Griffin's ordinary presentation is governed by a
finite point locus whose associated graded ideal yields $R_{n,\lambda,s}$.
We show that the equivariant quotient arises from a certain universal point arrangement in
the variables $\xxx_n$ and $\uuu_s$. More precisely, we study a linear subspace
arrangement
\[
X_{n,\lambda,s} \subset \QQ^{n+s}
\]
and prove that its vanishing ideal is generated by the same two families of
relations that appear in equivariant cohomology. This gives a conceptual
explanation for the appearance of the parameters $u_1,\dots,u_s$, and it produces
a free $\QQ[\uuu_s]$-basis indexed by certain $(n,\lambda,s)$-substaircase monomials.

The proof of the main theorem proceeds in three stages. First, using the affine
paving of $Y_{n,\lambda,s}$ due to Griffin-Levinson-Woo, we establish the
equivariant formality and freeness statements needed throughout the paper.
We then study the iterated projective bundle $Y_{n,\emptyset,s}$ to obtain the
first family of relations and a surjection onto $H_U^*(Y_{n,\lambda,s})$.
Second, following the strategy of Abe-Horiguchi, we map $Y_{n,\lambda,s}$ to
Grassmannians, identify its image inside a suitable Schubert variety, and use
vanishing of disjoint Schubert classes to obtain the second family of relations.
Finally, we compare the resulting quotient with the orbit-harmonics arrangement
$X_{n,\lambda,s}$. Both rings are free $\QQ[\uuu_s]$-modules of the same rank,
and the surjection is therefore an isomorphism.

Along the way, we also construct an action of $\symm_n$ on
$H_U^*(Y_{n,\lambda,s})$ by lifting the $\symm_n$-action on ordinary cohomology
constructed in \cite{GLW}. This equivariant lift allows us to pass from relations
in initial segments of the variables to relations for arbitrary subsets, just as
in the classical Springer-theoretic setting.

The paper is organized as follows. In Section~\ref{sec:background} we review the
necessary combinatorics, equivariant cohomology, equivariant Chern classes, and
orbit-harmonics background. Section~\ref{sec:orbitharmonics} studies the universal
arrangement $X_{n,\lambda,s}$ and proves the freeness results needed for the rank
comparison. Section~\ref{sec:geometry} establishes the $U$-action on
$Y_{n,\lambda,s}$ and proves the equivariant Borel presentation.

\subsection{Acknowledgements}

The author would like to thank, in no particular order, Sean Griffin, Brendon Rhoades, and Tomoo Matsumura for helpful conversations.

\section{Background}\label{sec:background}

\subsection{Combinatorics}

A \emph{partition} of $n$ is a sequence of positive integers $\lambda = (\lambda_1,\dots,\lambda_\ell)$ such that $\lambda_1 \geq \dots \geq \lambda_\ell$, and $\sum_i \lambda_i = n$. We express partitions with \emph{Young diagrams}, with $\ell$ rows of left-justified boxes, with the $i$th row containing $\lambda_i$ boxes. We will use \emph{English notation}, and so the rows will be arranged in weakly decreasing length from top to bottom. A \emph{filling} $P$ of a Young diagram is an assignment of nonnegative integers to each box of $\lambda$. 

A \emph{composition} of length $n$ is a tuple $(a_1,\dots,a_n) \in \ZZ_{\geq 0}^n$ of $n$ nonnegative integers. 

\subsection{Equivariant cohomology, localization, and equivariant Chern classes}

Computing equivariant cohomology is a central theme in Schubert calculus; see, for example, \cite{AF,GKM}. For this paper, we consider cohomology rings with $\QQ$-coefficients, i.e.
$H^*(X)=H^*(X;\QQ)$ for all spaces $X$. Let $T=(\CC^\times)^k$ be a complex
algebraic torus, and let $K=(S^1)^k\subset T$ be its maximal compact subgroup.
Since the inclusion $K\hookrightarrow T$ is a homotopy equivalence, we may
identify $T$-equivariant cohomology with the usual Borel equivariant cohomology
for the $K$-action. Thus, if $X$ is a space equipped with a $T$-action, we define
\[
H_T^*(X):=H^*(EK\times_K X).
\]
In particular,
\[
H_T^*(pt)=H^*(BT)\cong H^*(BK)\cong \QQ[t_1,\dots,t_k],
\qquad \deg t_i=2.
\]

Let $X$ be a complex algebraic variety. An \emph{affine paving} of $X$ is a
filtration by closed subvarieties
\[
\varnothing = X_0 \subset X_1 \subset \dots \subset X_m=X
\]
such that $X_i\setminus X_{i-1}$ is a finite disjoint union of affine spaces for
all $i$.

We will need the following standard lemma, which is a consequence of
\cite[Theorem~14.1]{GKM}.

\begin{lemma}\label{lem: aff-paving-induces-basis}
    Let $X$ be a complex algebraic variety which admits an algebraic action of a complex torus $T=(\CC^\times)^k$. Furthermore, suppose that $X$ has a $T$-invariant affine paving.
    Then $H^{\mathrm{odd}}(X)=0$. Consequently, $X$ is equivariantly formal, and
    there is a noncanonical isomorphism of graded $H_T^*(pt)$-modules
    \[
        H_T^*(X) \cong H_T^*(pt)\otimes_{\QQ} H^*(X).
    \]
    In particular, $H_T^*(X)$ is a free $H_T^*(pt)$-module of rank
    $\dim_{\QQ} H^*(X)$, and any additive basis of $H^*(X)$ lifts to an
    $H_T^*(pt)$-module basis of $H_T^*(X)$.
\end{lemma}

Affine pavings can also be used to show surjectivity results on maps between
cohomology rings. First, we recall an important commutative algebra fact.

\begin{lemma}[Graded Nakayama]\label{lem: graded-nakayama}
    Let $R = \bigoplus_{d \geq 0} R_d$ be a positively graded $\QQ$-algebra
    with $R_0 = \QQ$, and let
    \[
        R_+ := \bigoplus_{d > 0} R_d.
    \]
    Let $M$ be a finitely generated graded $R$-module, and let
    $N \subseteq M$ be a graded submodule. If
    \[
        M = N + R_+M,
    \]
    then $M = N$. In particular, if $M = R_+M$, then $M = 0$.
\end{lemma}

\begin{proof}
    Consider the graded quotient $\overline{M} := M/N$. Then $\overline{M}$ is
    a finitely generated graded $R$-module satisfying
    \[
        \overline{M} = R_+\overline{M}.
    \]
    Suppose $\overline{M} \neq 0$. Since $\overline{M}$ is finitely generated
    and $R$ is positively graded, there exists a smallest degree $r$ such that
    $\overline{M}_r \neq 0$. But every element of $R_+$ has positive degree, so
    \[
        \overline{M}_r = (R_+\overline{M})_r
        = \sum_{d>0} R_d\,\overline{M}_{r-d} = 0,
    \]
    since $\overline{M}_{r-d} = 0$ for all $d>0$ by minimality of $r$.
    This is a contradiction. Therefore $\overline{M} = 0$, so $M = N$.
\end{proof}

\begin{corollary}\label{cor: graded-nakayama-surj}
    Let $\varphi: M \to N$ be a graded homomorphism of finitely generated
    graded $R$-modules. If the induced map
    \[
        \overline{\varphi}: M/R_+M \longrightarrow N/R_+N
    \]
    is surjective, then $\varphi$ is surjective.
\end{corollary}

\begin{proof}
    Let $Q := \operatorname{coker}(\varphi)$. Then $Q$ is a finitely generated
    graded $R$-module, and the surjectivity of $\overline{\varphi}$ implies
    \[
        Q/R_+Q = 0.
    \]
    By Lemma~\ref{lem: graded-nakayama}, we have $Q = 0$. Hence $\varphi$ is
    surjective.
\end{proof}

\begin{remark}\label{rem: graded-nakayama-application}
    In this paper, we will apply Lemma~\ref{lem: graded-nakayama} and
    Corollary~\ref{cor: graded-nakayama-surj} with
    \[
        R = H_U^*(pt) \cong \QQ[u_1,\dots,u_s]
        \qquad\text{and}\qquad
        R_+ = (u_1,\dots,u_s).
    \]
\end{remark}

\begin{lemma}\label{lem: rel-paving-surj}\cite{GLW}
    Suppose $X$ is a smooth compact complex algebraic variety and
    $Y\subseteq X$ is a closed subvariety. If $Y$ and $X\setminus Y$ have affine
    pavings, then the restriction map
    \[
        H^*(X)\to H^*(Y)
    \]
    induced by the inclusion $Y\hookrightarrow X$ is surjective.
\end{lemma}

We will use the following equivariant enhancement.

\begin{lemma}\label{lem: rel-paving-equiv-surj}
    Let $X$ be a smooth compact complex algebraic variety, and let
    $Y\subseteq X$ be a closed $T$-stable subvariety, where
    $T=(\CC^\times)^k$. If $Y$ and $X\setminus Y$ admit $T$-stable affine
    pavings, then the restriction map
    \[
        H_T^*(X)\to H_T^*(Y)
    \]
    is surjective.
\end{lemma}

\begin{proof}
    Write $R:=H_T^*(pt)=\QQ[t_1,\dots,t_k]$ and
    $R_+:=\bigoplus_{d>0} H_T^{2d}(pt)$. Since $Y$ and $X\setminus Y$ have
    $T$-stable affine pavings, so does $X$. Hence $H^{\mathrm{odd}}(X)=0$ and
    $H^{\mathrm{odd}}(Y)=0$, so both $X$ and $Y$ are equivariantly formal by
    Lemma~\ref{lem: aff-paving-induces-basis}. In particular, $H_T^*(X)$ and
    $H_T^*(Y)$ are finitely generated free graded $R$-modules.

    By equivariant formality, the forgetful maps induce natural isomorphisms
    \[
        H_T^*(X)/R_+H_T^*(X)\cong H^*(X),
        \qquad
        H_T^*(Y)/R_+H_T^*(Y)\cong H^*(Y).
    \]
    Under these identifications, the map $H_T^*(X)\to H_T^*(Y)$ reduces modulo
    $R_+$ to the ordinary restriction map $H^*(X)\to H^*(Y)$, which is
    surjective by Lemma~\ref{lem: rel-paving-surj}. Since $H_T^*(Y)$ is a
    finitely generated graded $R$-module, graded Nakayama (Corollary \ref{cor: graded-nakayama-surj}) implies that
    $H_T^*(X)\to H_T^*(Y)$ is surjective.
\end{proof}

Now let $X$ be a paracompact Hausdorff space with a continuous $T$-action. A
\emph{$T$-equivariant complex vector bundle} on $X$ is a complex vector bundle
$E\to X$ together with a $T$-action on $E$ such that the projection map
$E\to X$ is $T$-equivariant and each $t\in T$ acts linearly on the fibers.
Applying the Borel construction produces a complex vector bundle
\[
E_T:=EK\times_K E \longrightarrow EK\times_K X=:X_T.
\]
The \emph{$i$th equivariant Chern class} of $E$ is defined by
\[
    c_i^T(E):=c_i(E_T)\in H^{2i}(X_T)=H_T^{2i}(X).
\]
These classes satisfy the same formal properties as ordinary Chern classes:
they are functorial under equivariant pullback, and for every short exact
sequence of $T$-equivariant vector bundles
\[
0\to E'\to E\to E''\to 0
\]
we have the Whitney sum formula
\[
    c^T(E)=c^T(E')\,c^T(E'').
\]
If $L$ and $M$ are $T$-equivariant line bundles, then
\[
    c_1^T(L\otimes M)=c_1^T(L)+c_1^T(M).
\]

For a character $\chi:T\to \CC^\times$, let $\CC_\chi$ denote the corresponding
one-dimensional representation of $T$. Then $\CC_\chi$ determines a
$T$-equivariant line bundle over a point, and its first equivariant Chern class
lies in $H_T^2(pt)$. If $\varepsilon_1,\dots,\varepsilon_k$ are the standard
characters of $(\CC^\times)^k$, we write
\[
    t_i:=c_1^T(\CC_{\varepsilon_i})\in H_T^2(pt),
\]
so that $H_T^*(pt)\cong \QQ[t_1,\dots,t_k]$. More generally, if a
$T$-representation $V$ decomposes as
\[
V\cong \CC_{\chi_1}\oplus \cdots \oplus \CC_{\chi_r},
\]
then
\[
    c^T(V)=\prod_{j=1}^r \bigl(1+c_1^T(\CC_{\chi_j})\bigr).
\]

We will also use the equivariant projective bundle formula. Let $E\to X$ be a
$T$-equivariant vector bundle of rank $r$, let $\pi:\PP(E)\to X$ be the
projective bundle of lines in $E$, and let $\mathcal O_E(-1)\subset \pi^*E$ be
the tautological line bundle. Setting
\[
x:=c_1^T(\mathcal O_E(-1)),
\]
we have
\[
H_T^*(\PP(E))
\cong
\frac{H_T^*(X)[x]}
{\left(c_r^T(E)-c_{r-1}^T(E)x+\cdots+(-1)^r x^r\right)}.
\]
Equivalently, if $\zeta:=c_1^T(\mathcal O_E(1))=-x$, then
\[
H_T^*(\PP(E))
\cong
\frac{H_T^*(X)[\zeta]}
{\left(\zeta^r-c_1^T(E)\zeta^{r-1}+\cdots+(-1)^r c_r^T(E)\right)}.
\]

Finally, if $Z\subseteq X$ is a $T$-stable closed subvariety of equivariant
codimension $d$, we write
\[
[Z]_T\in H_T^{2d}(X)
\]
for its equivariant fundamental class. We will use that equivariant Schubert classes are supported on the corresponding Schubert varieties; if $i: Z \hookrightarrow X$ is a $T$-equivariant inclusion and $Z$ is disjoint from a particular Schubert variety $X_\mu$, then the restricted class vanishes: $i^*[X_\mu]_T = 0$. 

\subsection{Grassmannians and Schubert varieties}

We follow the exposition of \cite{AH}. Given positive integers $k,n$, the complex \emph{Grassmannian} is the collection of $k$-dimensional subspaces of $\CC^n$:
\[
\Gr(k,n) := \{ V \subset \CC^n \mid \dim_\CC(V) = k \}
\]

Fix a complete flag of $\CC^n$, henceforth known as the \emph{reference flag}. We may define a cell decomposition of $\Gr(k,n)$ with respect to $U_\bullet$ indexed by partitions $\lambda$ which fit inside of a $k \times (n-k)$ rectangle. These cells have codimension given by $|\lambda|$, and their closures are referred to as the \emph{Schubert varieties} $X_\lambda(U_\bullet)$, defined by:

\[
X_\lambda(U_\bullet) := \{ V \in \Gr(k,n) \mid \dim_\CC(V \cap U_{n-k+i-\lambda_i}) \geq i \text{ for } 1 \leq i \leq k \}
\]

Let $F_\bullet, \widetilde{F}_\bullet$ be the standard reference flags defined by $F_i := \langle e_1,\dots,e_i \rangle$, and $\widetilde{F}_i := \langle e_n, \dots, e_{n-i+1} \rangle$. Then, we have the following lemma:

\begin{lemma}\label{lem: grass-intersect}
    Let $\lambda,\mu$ be partitions, and define $\mu^* := (n-k-\mu_k,\dots,n-k-\mu_1)$ to be the complement of $\mu$ in the $k \times (n-k)$ rectangle. Then, we have

    \[
X_\lambda(F_\bullet) \cap X_\mu(\widetilde{F}_\bullet) = \varnothing \text{ unless } \lambda \subset \mu^*
\]
\end{lemma}

\subsection{Partial Flag varieties and Springer Fibers}

Given a composition $\alpha := (\alpha_1,\dots,\alpha_k) \vDash n$, the type A \emph{partial flag variety} is defined to be space of flags
\[
\Fl_\alpha(\CC^n) = \{ V_\bullet = (V_0 \subset V_{\alpha_1} \subset V_{\alpha_1 + \alpha_2} \subset \dots \subset V_n) : \dim_\CC V_{\alpha_1 + \dots + \alpha_i} = \alpha_1 + \dots + \alpha_i\}
\]

where the jumps in dimension are given by $\alpha$. If $\alpha = (1^n)$, we refer to $\Fl_\alpha(\CC^n)$ as the \emph{complete flag variety} and write $\Fl_n(\CC^n) := \Fl_{1^n}(\CC^n)$, and if $m < n$, we write $\Fl_{1^m}(\CC^n) := \Fl_{(1^m,n-m)}(\CC^n)$.

A classical result of Borel states that the cohomology (presented here with rational coefficients) of the complete flag variety is isomorphic to the coinvariant algebra:
\[
H^*(\Fl_n(\CC^n)) \cong \QQ[\xxx_n]/\QQ[\xxx_n]^{\symm_n}_+ = R_n
\]

and so we refer to quotient ring presentations of cohomology rings as "Borel presentations." A standard equivariant analogue gives the presentation \cite[Chapter~4, Proposition~4.1]{AF} for $T = (\CC^\times)^n$-equivariant cohomology:
\[
H^*_T(\Fl_n) = \QQ[\xxx_n,\ttt_n]/(e_i(\xxx_n) - e_i(\ttt_n) : 1 \leq i \leq n)
\]

where $e_i(\xxx_n), e_i(\ttt_n)$ are the $i$th elementary symmetric functions in the $\xxx,\ttt$ variables respectively. We have similar presentations for partial flag varieties; see \cite{Borel1953}
for the ordinary-cohomology presentation and
\cite[Chapter~4, Corollary~5.4]{AF} for the equivariant one:
\[
H^*(\Fl_\alpha(\CC^n)) \cong \QQ[\xxx_n]^{\symm_{\alpha_1} \times \dots \times \symm_{\alpha_k}}/\QQ[\xxx_n]^{\symm_n}_+ = R_n^{\symm_{\alpha_1} \times \dots \times \symm_{\alpha_k}}
\]
\[
H^*_T(\Fl_\alpha) = \QQ[\xxx_n,\ttt_n]^{\symm_{\alpha_1} \times \dots \times \symm_{\alpha_k}}/(e_i(\xxx_n) - e_i(\ttt_n) : 1 \leq i \leq n)
\]
where $\symm_{\alpha_1} \times \dots \times \symm_{\alpha_k}$ acts only on the $\xxx$-variables.
We may ask that flags $V_\bullet$ be preserved by a linear operator $X: \CC^n \to \CC^n$. If $X$ is nilpotent with Jordan type $\lambda$, the \emph{nilpotent Springer fiber} is the set of flags
\[
\BBB_\lambda = \{ V_\bullet \in \Fl_n(\CC^n) : XV_i \subset V_i \}
\]
The cohomology ring of $\BBB_\lambda$ was first computed by DeConcini-Procesi, with relations later simplified by Tanisaki:
\[
H^*(\BBB_\lambda) \cong \QQ[\xxx_n]/( e_d(S): d > |S| - p^n_{|S|}(\lambda) )
\]
where $p^n_m(\lambda) = \lambda'_n + \dots + \lambda'_{n-m+1}$, and $\lambda'$ is the transpose of $\lambda$, padded with $0$'s to have length $n$. We refer to $I_\lambda = (e_d(S): d > |S|-p^n_{|S|}(\lambda))$ as the \emph{Tanisaki ideal}. 

The full torus $T=(\CC^\times)^n$ does not preserve $\BBB_\lambda$, but it is preserved by a subtorus 
\[
S = (\underbrace{u_1,\dots,u_1}_{\lambda_1},\dots,\underbrace{u_\ell,\dots,u_\ell}_{\lambda_\ell}) = (\CC^\times)^\ell \subset T
\]
Kumar and Procesi gave an algebro-geometric realization of the
$S$-equivariant cohomology of type $A$ Springer fibers \cite{KumarProcesi2012}; see also
Precup--Richmond \cite{PrecupRichmond2021} for an explicit equivariant basis. Abe and
Horiguchi \cite{AH} later gave a Borel-style presentation by generalizing the work of Tanisaki in \cite{T}:
\[
H^*_S(\BBB_\lambda) \cong \QQ[\xxx_n,\sss_n] / (e_d(\xxx_S|\sss_{\xi(n)}) : d > |S|-p^n_{|S|}(\lambda))
\]
Their argument consists of showing that the image of the Springer fiber $\BBB_\lambda$ under the map $\BBB_\lambda \to \Gr(d,n)$ sending $V_\bullet \mapsto V_d$ is contained in a certain Schubert variety $X_{\mu_0}$. It turns out that $e_d(\xxx_n|\sss_{\xi(n)})$ represents the closure of a certain Schubert cell $X_{|S|,d}$, which does not meet $X_{\mu_0}$. Pulling the relation back along the maps $H^*_T(\Gr(d,n)) \to H^*_S(\Gr(d,n)) \to H^*_S(\BBB_\lambda)$ implies the desired relations $e_d(\xxx_S|\sss_{\xi(n)})$ vanish in $H^*_S(\BBB_\lambda)$.

\subsection{The Rings $R_{n,\lambda,s}$ and the $\Delta$-Springer fibers}

In his thesis, Griffin introduced the \emph{$\Delta$-Springer modules} $R_{n,\lambda,s}$, which is a simultaneous generalization of the generalized coinvariant algebras $R_{n,k}$ of Haglund-Shimozono-Rhoades, as well as the Garsia-Procesi module $R_{\lambda}$. Given a partition $\lambda \vdash k \leq n$, and $s \geq \ell(\lambda)$, let
\[
R_{n,\lambda,s} := \frac{\QQ[\xxx_n]}{(x_1^s,\dots,x_n^s) + (e_d(S) : d > |S|-p^n_{|S|}(\lambda))}
\]

The $\Delta$-Springer module was realized as the cohomology ring of a variety
$Y_{n,\lambda,s}$ by Griffin-Levinson-Woo in \cite{GLW}; we briefly recall the
construction. Given a partition $\lambda\vdash k\leq n$ and $s\geq \ell(\lambda)$,
let
\[
\Lambda:=(n-k+\lambda_1,\dots,n-k+\lambda_s),
\]
where $\lambda_j=0$ if $j>\ell(\lambda)$. Let
\[
K:=|\Lambda|=s(n-k)+k,
\]
and let $N_\Lambda$ be a nilpotent operator of Jordan type $\Lambda$. Define the
\emph{$\Delta$-Springer fiber} to be the set of partial flags
\[
Y_{n,\lambda,s}:=\left\{V_\bullet\in \Fl_{1^n}(\CC^K)
\,\middle|\,
N_\Lambda V_i\subset V_i\text{ for }1\leq i\leq n,
\ N_\Lambda^{n-k}\CC^K\subset V_n\right\}.
\]
By \cite[Remark~3.2]{GLW}, since $N_\Lambda$ is nilpotent, the reduced
subvariety defined above is equal to the reduced subvariety cut out by the
conditions
\[
N_\Lambda V_i\subset V_{i-1}\qquad (1\leq i\leq n).
\]
We will freely use this equivalent shifted description when proving geometric
lemmas.
We will use a specific form of $N_\Lambda$ which will turn out to be more convenient. Let $P$ denote the filling of $\Lambda$ with the numbers $1,2,\dots,K$ from (1) top to bottom, (2) right to left; we shall refer to $P$ as the \emph{canonical filling} of $\Lambda$. 
\begin{example}\label{ex: P-diagram}
Let $\lambda = (2,2,1), n = 7$ and $s = 4$, we have $K = 4(2) + 5 = 13$, we have 
\[
P =
\begin{ytableau}
   10& 6& *(pink)3 & *(pink)1 \\
   11& 7& *(pink)4 & *(pink)2 \\
   12& 8& *(pink)5 \\
   13& 9
\end{ytableau}
\]
\end{example}
We will refer to $[\lambda]$ (shaded in pink) as the boxes forming shape $\lambda$ containing the numbers $1,\dots,k$. Fixing a reference basis $e_1,\dots,e_K$ of $\CC^K$, we define the nilpotent element $N_\Lambda$ to be
\[
N_\Lambda(e_i) = \begin{cases}
    e_j & , j\text{ is directly to the right of }i \\
    0 & , i\text{ is on the right edge of }\Lambda
\end{cases}
\]
For instance, in the running example, we have:
\[
N_\Lambda: \begin{cases}
    e_{10} \mapsto e_6 \mapsto e_3 \mapsto e_1 \mapsto 0 \\
    e_{11} \mapsto e_7 \mapsto e_4 \mapsto e_2 \mapsto 0 \\
    e_{12} \mapsto e_8 \mapsto e_5 \mapsto 0 \\
    e_{13} \mapsto e_9 \mapsto 0
\end{cases}
\]

We define the function $\xi: [K] \to [s]$ which keeps track of the rows of $P$, which will be convenient later. Let $P$ denote the canonical filling of $\Lambda$, as above. Writing $P_i = \{a_1,\dots,a_{\Lambda_i}\}$ for the set of entries in the $i$th row of $P$, we define:
\[
\xi(i) = j \Leftrightarrow i \text{ occurs in }P_j
\]
We define the torus
\[
U:=\left\{\operatorname{diag}(u_{\xi(1)},\dots,u_{\xi(K)}):(u_1,\dots,u_s)\in
(\CC^\times)^s\right\}\subset (\CC^\times)^K.
\]
Thus $U$ acts diagonally on $\CC^K$, with all basis vectors in the same row of
$P$ having the same weight.

For later use, fix a permutation $w_U\in \symm_K$ such that
the flag $U_\bullet:=w_UF_\bullet$ refines
\[
\cdots\subset N_\Lambda^2\CC^K\subset N_\Lambda\CC^K\subset \CC^K.
\]
\begin{definition}\label{def: phi-lambda}
We define
\[
\phi_\lambda:[n]\to [s],\qquad \phi_\lambda(r):=\xi(w_U(r)).
\]
Equivalently, $\phi_\lambda=(\xi\circ w_U)|_{[n]}$.
\end{definition}
In particular, the restriction of equivariant Chern classes to the fixed flag
indexed by $w_U$ sends
\[
x_r\longmapsto u_{\xi(w_U(r))}=u_{\phi_\lambda(r)}.
\]
More generally, if a $U$-fixed point is indexed by a word $w$, then the induced
restriction map sends $x_i$ to $u_{\xi(w(i))}$ and fixes the classes $u_j$.

\subsection{Commutative algebra and orbit harmonics}

We now recall the orbit harmonics philosophy in the form needed for this paper.
Let $X \subset \QQ^n$ be a finite set, and let
\[
I(X):=\{f\in \QQ[\xxx_n] : f(x)=0 \text{ for all } x\in X\}
\]
be its vanishing ideal. Since $X$ is finite, evaluation identifies the quotient
$\QQ[\xxx_n]/I(X)$ with the ring of $\QQ$-valued functions on $X$. In particular,
\[
\dim_\QQ \QQ[\xxx_n]/I(X)=|X|.
\]

The key observation is that one may pass from the generally inhomogeneous ideal $I(X)$
to a homogeneous ideal without changing the vector space dimension of the quotient.
For a nonzero polynomial
\[
f=f_d+f_{d-1}+\cdots+f_0
\]
with each $f_i$ homogeneous of degree $i$ and $f_d\neq 0$, define the
\emph{top-degree part} of $f$ by $\tau(f):=f_d$. We then set
\[
T(X):=\langle \tau(f) : f\in I(X),\ f\neq 0\rangle.
\]
Equivalently, $T(X)=\gr I(X)$ is the associated graded ideal of $I(X)$ with respect
to the total-degree filtration. The quotient $\QQ[\xxx_n]/T(X)$ is therefore graded,
and one has
\[
\dim_\QQ \QQ[\xxx_n]/T(X)=\dim_\QQ \QQ[\xxx_n]/I(X)=|X|.
\]

When the point locus $X$ is stable under the action of a finite group $G$ acting linearly
on $\QQ^n$, both $I(X)$ and $T(X)$ are $G$-stable. It follows that the graded quotient
$\QQ[\xxx_n]/T(X)$ inherits a graded $G$-module structure. The passage
\[
X \rightsquigarrow \QQ[\xxx_n]/T(X)
\]
is often referred to as \emph{orbit harmonics}. One may also think of taking $\gr$ of the ideal $I(X)$ as a flat deformation of the point locus $X$ to a nonreduced scheme supported at the origin.

In many examples, the ideal $T(X)$ admits an explicit description by generators, and
this turns a concrete point set into a graded quotient ring with representation-theoretic
or geometric significance. In our setting, Griffin showed that the ordinary
$\Delta$-Springer ring $R_{n,\lambda,s}$ may be obtained from a finite point locus
$X_{n,\lambda,s}(\aalpha_s)\subset \QQ^n$ by the rule
\[
R_{n,\lambda,s}\cong \QQ[\xxx_n]/\gr I(X_{n,\lambda,s}(\aalpha_s)).
\]
This point-locus description should be viewed as the nonequivariant shadow of the
constructions used later in this paper.

Our equivariant presentation is guided by the same philosophy. We enlarge the finite
locus to a universal locus
\[
X_{n,\lambda,s}\subset \QQ^{n+s}
\]
whose extra coordinates record the torus parameters. The ideal
$J^\QQ_{n,\lambda,s}\subset \QQ[\xxx_n,\uuu_s]$ will be defined by explicit vanishing
conditions on this universal locus. Proving that
\[
J^\QQ_{n,\lambda,s}=I(X_{n,\lambda,s})
\]
will imply that the quotient $\QQ[\xxx_n,\uuu_s]/J^\QQ_{n,\lambda,s}$ is a free
$\QQ[\uuu_s]$-module with basis indexed by the same substaircase monomials that appear
in the ordinary theory. 

We will make use of the following standard lemma.
\begin{lemma}\label{lem: basis-gr-imply-basis}
    Let $I \subset \QQ[\xxx_n]$ be an ideal, and $\BBB$ a family of homogeneous elements of $\QQ[\xxx_n]$ which descends to a $\QQ$-basis of $\QQ[\xxx_n]/\gr I$. Then, $\BBB$ descends to a $\QQ$-basis of $\QQ[\xxx_n]/I$.
\end{lemma}

\section{Quotient Rings and Orbit Harmonics}\label{sec:orbitharmonics}

\subsection{The finite locus $X_{n,\lambda,s}(\boldmath{\alpha}_s)$}

Fix positive integers $n,k,s$, with $k \leq n$, and let $\lambda$ be a partition such that $\ell(\lambda) \leq s$. Let $\aalpha_s := (\alpha_1,\dots,\alpha_s) \in \QQ$, $\alpha_i \neq \alpha_j$ for all $i,j$ be a tuple of $s$ distinct rational numbers. Then, the point locus $X_{n,\lambda,s}$ is defined to be
\[
X_{n,\lambda,s}(\aalpha_s) := \{ (p_1,\dots,p_n) \in \QQ^n : p_i \in \{\alpha_1,\dots,\alpha_s\} \text{ for } 1 \leq i \leq n, \#\alpha_i \geq \lambda_i \text{ for } 1 \leq i \leq s \}
\]
where $\#\alpha_i$ is the number of times $\alpha_i$ appears amongst $p_1,\dots,p_n$. We have the \emph{defining ideal} of $X_{n,\lambda,s}(\aalpha_s)$:
\[
I(X_{n,\lambda,s}(\aalpha_s)) := \{ f \in \QQ[\xxx_n] : f(p) = 0 \text{ for } p \in X_{n,\lambda,s}\} \subset \QQ[\xxx_n]
\]
We record here a few results of Griffin, which we will use later on.
\begin{theorem}\cite{G}
    We have $I_{n,\lambda,s} = \gr I(X_{n,\lambda,s}(\aalpha_s))$, so that $R_{n,\lambda,s} = \QQ[\xxx_n]/\gr I(X_{n,\lambda,s}(\aalpha_s))$. Furthermore, we have an isomorphism of $\symm_n$-modules $R_{n,\lambda,s} \cong_{\symm_n} \QQ X_{n,\lambda,s}(\aalpha_s)$.
\end{theorem}

\begin{definition}
    Let $\mathbf{a}^{(1)},\dots,\mathbf{a}^{(\ell)}$ be compositions of length $\lambda_1,\dots,\lambda_\ell$. A \emph{shuffle} of the compositions $\mathbf{a}^{(1)},\dots,\mathbf{a}^{(\ell)}$ is a composition $\mathbf{a} = (a_1,\dots,a_{\lambda_1 + \dots + \lambda_\ell})$ such that $\mathbf{a}$ can be decomposed into disjoint subsequences $\mathbf{a}^{(1)},\dots,\mathbf{a}^{(\ell)}$.
\end{definition}

\begin{definition}\cite{G}\label{def: nls-substaircases}
    For $1 \leq j \leq \lambda_1$, let $\beta^j(\lambda) = (0,1,\dots,\lambda'_j -1)$. An \emph{$(n,\lambda,s)$ staircase} is a shuffle $\mathbf{\gamma} = (\gamma_1,\dots,\gamma_n)$ of the compositions $\beta^1(\lambda),\dots,\beta^{\lambda_1}(\lambda)$ and $((s-1)^{n-k})$. Define the set of \emph{$(n,\lambda,s)$-substaircases} to be the set of compositions
    \[
\CCC_{n,\lambda,s} := \{ \mathbf{a} = (a_1,\dots,a_n) : 0 \leq a_i \leq \gamma_i \text{ for some } (n,\lambda,s)\text{-staircase } \mathbf{\gamma} \}
\]
    and the set of $(n,\lambda,s)$-substaircase monomials to be
    \[
\AAA_{n,\lambda,s} := \{ \xxx^\mathbf{a} := x_1^{a_1}\dots x_n^{a_n} : \mathbf{a} \in \CCC_{n,\lambda,s} \}
\]
\end{definition}

\begin{theorem}\cite{G}\label{thm: griffin-basis}
    The set $\AAA_{n,\lambda,s}$ of $(n,\lambda,s)$-substaircase monomials forms a $\QQ$-basis of $R_{n,\lambda,s}$.
\end{theorem}

\subsection{The linear subspace arrangement $X_{n,\lambda,s}$}

We now consider the universal locus over all tuples $(\alpha_1,\dots,\alpha_s)$.
Let $\QQ[\xxx_n,\uuu_s]$ denote the coordinate ring of $\QQ^{n+s}$. Define
$X_{n,\lambda,s}\subset \QQ^{n+s}$ to be the locus of tuples
\[
(\beta_1,\dots,\beta_n;\alpha_1,\dots,\alpha_s)
\]
such that
\begin{itemize}
    \item $\beta_j\in\{\alpha_1,\dots,\alpha_s\}$ for all $1\leq j\leq n$, and
    \item for each $1\leq i\leq s$, the value $\alpha_i$ occurs among
    $\beta_1,\dots,\beta_n$ with multiplicity at least $\lambda_i$.
\end{itemize}
Equivalently,
\[
X_{n,\lambda,s}=
\overline{\bigcup_{\aalpha_s}
X_{n,\lambda,s}(\aalpha_s)\times \{\aalpha_s\}},
\]
where the union is over tuples $\aalpha_s=(\alpha_1,\dots,\alpha_s)$ with
pairwise distinct coordinates.

\begin{definition}\cite{AH}\label{defn: double-e}
    Given two families of indeterminates $x_1,\dots,x_m$ and $a_1,a_2,\dots$, define the \emph{double Tanisaki polynomial} to be
    \begin{equation}\label{eq: def-double-e}
        e_d(x_1,\dots,x_m | a_1,a_2,\dots ) := \sum_{r = 0}^d (-1)^{d-r} e_r(x_1,\dots,x_m)h_{d-r}(a_1,\dots,a_{m+1-d})
    \end{equation}
    for $d \geq 0$, and $e_i, h_i$ are the $i$th elementary and complete homogeneous symmetric polynomials respectively.
\end{definition}

\begin{definition}\label{def: ideal-jqnls}
    Let $J^\QQ_{n,\lambda,s}\subset \QQ[\xxx_n,\uuu_s]$ be the ideal generated by
    \begin{itemize}
        \item $(x_i-u_1)\cdots (x_i-u_s)$ for $1\leq i\leq n$, and
        \item $e_d(x_{i_1},\dots,x_{i_m}\mid u_{\phi_\lambda(1)},\dots,u_{\phi_\lambda(n)})$
        for $1\leq m\leq n$, $1\leq i_1<\cdots<i_m\leq n$, and
        $d>m-p_m^n(\lambda)$.
    \end{itemize}
\end{definition}

\begin{lemma}\label{lem: ideal-containment}
    We have the containment of ideals
    \[
    J^\QQ_{n,\lambda,s}\subset I(X_{n,\lambda,s}).
    \]
\end{lemma}

\begin{proof}
    Let $(\beta_1,\dots,\beta_n;\alpha_1,\dots,\alpha_s)\in X_{n,\lambda,s}$.
    Since each $\beta_i$ is one of the $\alpha_j$, the polynomial
    $(x_i-u_1)\cdots(x_i-u_s)$ vanishes at this point for every $i$, so the
    first family of generators lies in $I(X_{n,\lambda,s})$.

    For the second family, recall that
    \[
    e_d(x_1,\dots,x_m\mid a_1,a_2,\dots)
    =\sum_{r=0}^d (-1)^{d-r} e_r(x_1,\dots,x_m)h_{d-r}(a_1,\dots,a_{m+1-d}),
    \]
    which is the coefficient of $q^d$ in
    \[
    \frac{(1+x_1q)\cdots(1+x_mq)}{(1+a_1q)\cdots(1+a_{m+1-d}q)}.
    \]

    Fix indices $1\leq i_1<\cdots<i_m\leq n$, and let $c_{r,m}$ be the number of
    boxes in the $r$th row of $\lambda$ weakly to the right of the
    $(n-m+1)$st column of the padded diagram; equivalently,
    \[
    c_{r,m}=\max(\lambda_r-(n-m),0).
    \]
    Then
    \[
    \sum_{r=1}^s c_{r,m}=p_m^n(\lambda).
    \]
    Since each $\alpha_r$ occurs among $\beta_1,\dots,\beta_n$ at least
    $\lambda_r$ times, omitting $n-m$ of the $\beta$'s can remove at most $n-m$
    copies of $\alpha_r$. Hence among the selected coordinates
    $\beta_{i_1},\dots,\beta_{i_m}$ there are still at least $c_{r,m}$ copies of
    $\alpha_r$.

    Therefore
    \[
    \frac{(1+\beta_{i_1}q)\cdots(1+\beta_{i_m}q)}
    {(1+\alpha_1 q)^{c_{1,m}}\cdots (1+\alpha_s q)^{c_{s,m}}}
    \]
    is a polynomial in $q$ of degree $m-p_m^n(\lambda)$. Since
    $d>m-p_m^n(\lambda)$, we have $p_m^n(\lambda)\geq m+1-d$, and by the
    definition of $\phi_\lambda$ the product
    \[
    \prod_{r=1}^{m+1-d} (1+u_{\phi_\lambda(r)}q)
    \]
    divides
    \[
    (1+u_1q)^{c_{1,m}}\cdots (1+u_sq)^{c_{s,m}}.
    \]
    Hence
    \[
    \frac{(1+x_{i_1}q)\cdots(1+x_{i_m}q)}
    {\prod_{r=1}^{m+1-d}(1+u_{\phi_\lambda(r)}q)}
    \]
    has degree at most $d-1$ after evaluation on $X_{n,\lambda,s}$, so its
    $q^d$-coefficient vanishes. This shows that
    \[
    e_d(x_{i_1},\dots,x_{i_m}\mid u_{\phi_\lambda(1)},\dots,u_{\phi_\lambda(n)})
    \in I(X_{n,\lambda,s}).
    \]
\end{proof}

\begin{definition}
    For a set $\AAA$ of polynomials, define $\AAA^\uuu$ to be the set of polynomials given by
    \[
\AAA^\uuu := \{ f \cdot u_1^{a_1}\dots u_s^{a_s} : f \in \AAA, a_1,\dots,a_s \geq 0\}
\]
\end{definition} 

\begin{proposition}\label{prop: module-dim}
    We have the equality of ideals $J^\QQ_{n,\lambda,s} = I(X_{n,\lambda,s})$. The set $\AAA^\uuu_{n,\lambda,s}$ descends to a $\QQ$-basis for the common quotient ring
    \[
\QQ[\xxx_n,\uuu_s]/J^\QQ_{n,\lambda,s} = \QQ[\xxx_n,\uuu_s]/I(X_{n,\lambda,s})
\]
    As a corollary, we have that $\QQ[\xxx_n,\uuu_s]/J^\QQ_{n,\lambda,s}$ is a free $\QQ[\uuu_s]$-module of rank $|\AAA_{n,\lambda,s}|$.
\end{proposition}

\begin{proof}
    We adapt the argument given in Chou-Matsumura-Rhoades \cite{CMR}.
    
    We first establish linear independence of $\AAA^\uuu_{n,\lambda,s}$ modulo the ideal $I(X_{n,\lambda,s})$. By lemma \ref{lem: ideal-containment}, this will imply linear independence of $\AAA^\uuu_{n,\lambda,s}$ modulo $J^\QQ_{n,\lambda,s}$.
    Assume the contrary, and that $\AAA^\uuu_{n,\lambda,s}$ is not linearly independent modulo $I(X_{n,\lambda,s})$. Then, there exist polynomials $f_a(\uuu_s) \in \QQ[\uuu_s]$ such that
    \[
\sum_{a \in \AAA_{n,\lambda,s}} f_a(\uuu_s) \cdot a \in I(X_{n,\lambda,s})
\]
    If $\aalpha_s \in \QQ^s$ has distinct coordinates, we may apply the evaluation map $\varepsilon_{\aalpha_s}: u_i \mapsto \alpha_i$ to obtain
    \[
\sum_{a \in \AAA_{n,\lambda,s}} f_a(\aalpha_s) \cdot a \in I(X_{n,\lambda,s}(\aalpha_s))
\]
    by Theorem \ref{thm: griffin-basis}, and Lemma \ref{lem: basis-gr-imply-basis}, we must have that $\AAA_{n,\lambda,s}$ is linearly independent modulo $I(X_{n,\lambda,s}(\aalpha_s))$. Since the collection of tuples $\aalpha_s \in \QQ^s$ with distinct coordinates is dense, we conclude that $f_a(\uuu_s) = 0$ for all $a \in \AAA_{n,\lambda,s}$. This proves that $\AAA^\uuu_{n,\lambda,s}$ is linearly independent modulo $I(X_{n,\lambda,s})$.

    Next, we prove that $\AAA^\uuu_{n,\lambda,s}$ spans $\QQ[\xxx_n,\uuu_s]/J^\QQ_{n,\lambda,s}$, which will imply by lemma \ref{lem: ideal-containment} that $\AAA^\uuu_{n,\lambda,s}$ spans $\QQ[\xxx_n,\uuu_s]/I(X_{n,\lambda,s})$. Consider a monomial $m = m_\xxx m_\uuu \in \QQ[\xxx_n,\uuu_s]$, where $m_\xxx \in \QQ[\xxx_n], m_\uuu \in \QQ[\uuu_s]$. We induct on the $\xxx$-degree of $m_\xxx$.
    If $\deg(m_\xxx) = 0$, then $m \in \QQ[\uuu_s]$, and this is clear, since $1 \in \AAA_{n,\lambda,s}$. If $\deg(m_\xxx) > 0$, by theorem \ref{thm: griffin-basis}, we may write
    \begin{equation}\label{eq: mx-expansion}
        m_\xxx = \sum_{a \in \AAA_{n,\lambda,s}} c_a \cdot a + \sum_{g} h_g \cdot \overline{g}
    \end{equation}
    where $c_a \in \QQ$, $g \in J^\QQ_{n,\lambda,s}$ denotes a generator given in Definition \ref{def: ideal-jqnls}, and $\overline{g} =g(\xxx_n,0)$ the result of evaluating the $u$-variables in $g$ at $0$. Since equation \ref{eq: mx-expansion} does not involve $\uuu_s$, we may assume that $c_a = 0$ if $\deg_\xxx(a) \neq \deg_\xxx(m_\xxx)$, and that $h_g \cdot \overline{g}$ is homogeneous of degree $\deg_\xxx(m_\xxx)$. We then multiply through by $m_\uuu$ to obtain
    \begin{align}
    m = m_\xxx m_\uuu & = \sum_{a \in \AAA_{n,\lambda,s}} c_a \cdot a \cdot m_\uuu + \sum_{g} h_g \cdot \overline{g} \cdot m_\uuu \\
    & = \sum_{a \in \AAA_{n,\lambda,s}} c_a \cdot a \cdot m_\uuu + \sum_{g} h_g \cdot g \cdot m_\uuu + \sum_{g} h_g \cdot (\overline{g} - g) \cdot m_\uuu
    \end{align}
    noting that $h_g \cdot g \cdot m_\uuu \in J^\QQ_{n,\lambda,s}$, $ a \cdot m_\uuu \in \AAA^\uuu_{n,\lambda,s}$, and that $\deg_\xxx(\overline{g} - g) < \deg_\xxx(m_\xxx)$, the claim follows by induction.

    This implies the natural surjection
    \[
    \QQ[\xxx_n,\uuu_s]/J_{n,\lambda,s}^\QQ \twoheadrightarrow \QQ[\xxx_n,\uuu_s]/I(X_{n,\lambda,s})
    \]
    is an isomorphism, since both quotients have the same basis.
\end{proof}

\section{Borel Presentation of $H^*_U(Y_{n,\lambda,s})$}\label{sec:geometry}

To make sense of a discussion of $H^*_U(Y_{n,\lambda,s})$, we must first prove that the torus $U$ acts on $Y_{n,\lambda,s}$.

\begin{lemma}\label{lem:UactsonYnls}
    The torus $U$ acts on the variety $Y_{n,\lambda,s}$.
\end{lemma}

\begin{proof}[Proof of Lemma~\ref{lem:UactsonYnls}]
    Let $u\in U$. We first show that $u$ commutes with $N_\Lambda$.
    Since both are linear maps on $\CC^K$, it suffices to check this on the
    basis vectors $e_1,\dots,e_K$. Fix a row $j$ of the tableau $P$, and let
    $i_1>i_2>\cdots>i_\ell$ be the entries in the $j$th row of $P$. Then
    $N_\Lambda(e_{i_r})=e_{i_{r+1}}$ whenever $i_r$ is not on the right edge of
    $\Lambda$, and $u(e_{i_r})=u_j e_{i_r}$ for all $r$ because $u$ is constant
    on rows. Hence
    \[
    uN_\Lambda(e_{i_r})=u(e_{i_{r+1}})=u_j e_{i_{r+1}}
    =N_\Lambda(u_j e_{i_r})=N_\Lambda u(e_{i_r}).
    \]
    Thus $uN_\Lambda=N_\Lambda u$.

    Now let $V_\bullet\in Y_{n,\lambda,s}$. Using the equivalent shifted
    description of $Y_{n,\lambda,s}$ recalled above, we have
    $N_\Lambda V_i\subset V_{i-1}$ for $1\leq i\leq n$, and also
    $N_\Lambda^{n-k}\CC^K\subset V_n$. Therefore
    \[
    N_\Lambda(uV_i)=uN_\Lambda(V_i)\subset uV_{i-1}
    \]
    and
    \[
    N_\Lambda^{n-k}\CC^K
    =N_\Lambda^{n-k}(u\CC^K)
    =uN_\Lambda^{n-k}\CC^K
    \subset uV_n.
    \]
    Hence $uV_\bullet\in Y_{n,\lambda,s}$, so $U$ acts on $Y_{n,\lambda,s}$.
\end{proof}

We state our main theorem:

\begin{theorem}\label{thm: main}
    Let the ideal $J_{n,\lambda,s} \subset \QQ[\xxx_n,\uuu_s]$ be generated by the relations 
    \begin{itemize}
        \item $(x_i - u_1)\dots(x_i - u_s)$ for $1 \leq i \leq n$
        \item $e_d(x_{i_1},\dots,x_{i_m} | u_{\phi_\lambda(1)},\dots,u_{\phi_\lambda(n)})$ for $1 \leq m \leq n$, $1 \leq i_1 < \dots < i_m \leq n$, and $d > m - p^n_m(\lambda)$
    \end{itemize}
    Then there is a ring isomorphism
    \begin{equation}\label{eq: maintheorem}
        H^*_U(Y_{n,\lambda,s}) \cong \QQ[\xxx_n,\uuu_s]/J_{n,\lambda,s}
    \end{equation}
\end{theorem}
We refer to the relations for $J_{n,\lambda,s}$ in the first bullet point as \emph{family 1}, and the relations in the second as \emph{family 2}. Note that evaluating $u_1 = \dots = u_s = 0$ recovers the Griffin ideal $I_{n,\lambda,s}$.

\subsection{Affine pavings}

We first recount the description of the affine paving of $Y_{n,\lambda,s}$ defined by Griffin-Levinson-Woo in \cite{GLW}. An application of a few standard lemmas give the $H^*_U(\bullet)$-module dimension of $H^*_U(Y_{n,\lambda,s})$.

\begin{definition}\cite{GLW}\label{def: P-admissible}
    Given an injective map $w: [n] \to [K]$, we say that $w$ is \emph{admissible} with respect to a tableau $P$ if the following hold:
    \begin{itemize}
        \item $[k] \subset \im(w)$
        \item For $i \leq n$, if $w(i) = P(a,b)$ for $(a,b) \not\in \edge[\Lambda]$, then $P(a,b+1) \in \{w(1),\dots,w(i-1)\}$. 
    \end{itemize}
\end{definition}

The collection of admissible injective maps give a criteria for which Schubert cells $C_w$ of $\Fl_{(1^n)}(\CC^K)$ have nonempty intersection with $Y_{n,\lambda,s}$. In fact, we have more:

\begin{theorem}\cite{GLW}
    The cells $\{ C_w \cap Y_{n,\lambda,s} : w \text{ is admissible with respect to }P\}$ form an affine paving of $Y_{n,\lambda,s}$.
\end{theorem}

We will also need the following fact:

\begin{lemma}\cite{GLW}\label{lem: wsandAnls}
    We have that $|\{w \text{ is admissible with respect to }P\}| = |\AAA_{n,\lambda,s}|$.
\end{lemma}

\begin{proof}
    This follows from the fact that $\Hilb(H^*(Y_{n,\lambda,s}),q) = \Hilb(R_{n,\lambda,s},q)$, which was proven in \cite{GLW}. Evaluating $q = 1$, we obtain the result.
\end{proof}

\begin{lemma}
    The affine paving given by the cells $\{ C_w \cap Y_{n,\lambda,s} : w \text{ is admissible with respect to } P\}$ is stable under the action of $U$.
\end{lemma}

\begin{proof}
    Let $w$ be admissible with respect to $P$. Since $U \subset T$, and the Schubert cells $C_w$ are stable under the action of $T$, then $UC_w \subset C_w$. By lemma \ref{lem:UactsonYnls}, we have that $UY_{n,\lambda,s} \subset Y_{n,\lambda,s}$. Therefore, given $V_\bullet \in C_w \cap Y_{n,\lambda,s}$ and $u \in U$, we have that $uV_\bullet \in C_w$ and $uV_\bullet \in Y_{n,\lambda,s}$. The claim is proven.
\end{proof}

\begin{corollary}\label{cor: moduledim2}
    The ring $H^*_U(Y_{n,\lambda,s})$ is a free $\QQ[\uuu_s]$-module of rank $|\AAA_{n,\lambda,s}|$.
\end{corollary}

\begin{proof}
    By lemma \ref{lem: aff-paving-induces-basis}, we have that $H^*_U(Y_{n,\lambda,s})$ is a free module of rank 
    \[
| \{w : w \text{ is admissible with respect to }P\}|
\]
    Applying lemma \ref{lem: wsandAnls} proves the second half of the claim.
\end{proof}

\subsection{Family 1 and the iterated projective bundle $Y_{n,\emptyset,s}$}

To prove the first family of relations in $H^*_U(Y_{n,\lambda,s})$, we follow the argument in \cite{GLW} by considering the iterated projective bundle $Y_{n,\emptyset,s}$, where $\emptyset$ is the empty partition. We briefly recall the construction of $Y_{n,\emptyset,s}$ in \cite{GLW}. Let $\widetilde{N_\Lambda^{-1}V}_{n-1}$ denote the rank $s+n-1$ vector bundle over $Y_{n-1,\emptyset,s}$ whose fiber over $V_\bullet$ is $N_\Lambda^{-1}(V_{n-1})$, and $\widetilde{V}_{n-1}$ denote the rank $n-1$ tautological subspace bundle over $Y_{n-1,\emptyset,s}$.

\begin{lemma}\cite{GLW}\label{lem: iterated-proj-bundle}
    Let $P'$ denote the result of deleting the first column of $P$. Then, the map
    \[
Y_{n,\emptyset,s} \to Y_{n-1,\emptyset,s}
\]
    obtained by $(V_0 \subset \dots \subset V_n) \mapsto (V_0 \subset \dots \subset V_{n-1})$ is a $\PP^{s-1}$-bundle map. Furthermore, we have an isomorphism
    \begin{equation}\label{eq: ps-1bundle-map}
     Y_{n,\emptyset,s} \to \PP(\widetilde{N_\Lambda^{-1}V}_{n-1}/\widetilde{V}_{n-1})
    \end{equation}
    defined by mapping $V_\bullet$ to $V_n/V_{n-1}$ over $(V_0 \subset \dots \subset V_{n-1})$.
\end{lemma}

\begin{lemma}\label{lem: proj-bundle-cohom}
    We have that
    \[
H^*_U(Y_{n,\emptyset,s}) \cong \frac{\QQ[x_1,\dots,x_n;u_1,\dots,u_s]}{( (x_i-u_1)\dots(x_i-u_s) : 1 \leq i \leq n)}
\]
\end{lemma}

\begin{proof}
    We will instead prove that $(u_1-x_i)\dots(u_s-x_i) = 0$. By Viete's relations, we may rewrite the relation as
    \begin{equation}\label{eq: rel-family1-chern}
        (u_1-x_i)\dots(u_s-x_i) = e_s(\uuu_s) - e_{s-1}(\uuu_s)x_i + \dots + (-1)^s x_i^s
    \end{equation}
    We will proceed by induction, as in \cite{GLW}. For $n = 1$, note that $Y_{1,\emptyset,s}$ is $\PP^{s-1}$. Letting $\CC^s$ denote the trivial $U$-bundle over $Y_{n,\emptyset,s}$, we see that
    \[
c^U(\CC^{s}) = (1+u_1)\dots ( 1 + u_s) = 1 + e_1(\uuu_s) + \dots + e_s(\uuu_s)
\]
    so that $c_i^U(\CC^s) = e_i(\uuu_s)$. The tautological line bundle $\ell_1$ is identified with $\mathcal{O}(-1)$, and letting $x_1 = c_1^U(\mathcal{O}(-1))$, the $U$-equivariant cohomology of $\PP^{s-1}$ is given by $\QQ[x_1,\uuu_s]/(e_s(\uuu_s) - e_{s-1}(\uuu_s)x_1 + \dots + (-1)^s x_1^s)$. 

    Next, suppose by way of induction that 
    \[
H^*_U(Y_{n-1,\emptyset,s}) \cong \frac{\QQ[x_1,\dots,x_{n-1};\uuu_s]}{( e_s(\uuu_s) - e_{s-1}(\uuu_s)x_i + \dots + (-1)^s x_i^s : 1 \leq i \leq n-1)}
\]
    Let $E$ denote the rank-$s$ vector bundle $\widetilde{N_\Lambda^{-1}V}_{n-1}/\widetilde{V}_{n-1}$. Since $Y_{n,\emptyset,s} \cong \PP(E)$ by lemma \ref{lem: iterated-proj-bundle}, we have that $\widetilde{V}_n/\widetilde{V}_{n-1} \cong \mathcal{O}_E(-1)$, and by the (equivariant) projective bundle formula, we have
    \[
H^*_U(Y_{n,\emptyset,s}) \cong \frac{H^*_U(Y_{n-1,\emptyset,s})[x_n]}{\langle c_s^U(E) - c_{s-1}^U(E)x_n + \dots + (-1)^sx_n^s \rangle }
\]
    It now suffices to show that $c_i^U(E)=e_i(\uuu_s)$ for all $1\leq i\leq s$.
From \cite{GLW}, we have a short exact sequence of $U$-equivariant vector
bundles
\[
\begin{tikzcd}
0 \arrow[r] & E \arrow[r] & \CC^{ns}/\widetilde{V}_{n-1}
\arrow[r] & \im(N_\Lambda)/\widetilde{V}_{n-1} \arrow[r] & 0,
\end{tikzcd}
\]
so by Whitney sum formula,
\[
c^U(E)=c^U(\CC^{ns}/\im(N_\Lambda)).
\]
Now $\CC^{ns}/\im(N_\Lambda)$ is spanned by the images of the first-column basis
vectors, one from each row of $P$, and hence is isomorphic as a
$U$-representation to
\[
\CC_{u_1}\oplus \cdots \oplus \CC_{u_s}.
\]
Therefore
\[
c^U(E)=\prod_{j=1}^s (1+u_j)=1+e_1(\uuu_s)+\cdots+e_s(\uuu_s),
\]
which proves that $c_i^U(E)=e_i(\uuu_s)$.

\end{proof}
We have the following lemma:
\begin{lemma}\label{lem: proj-bundle-surjection}
    The closed embedding $Y_{n,\lambda,s} \hookrightarrow Y_{n,\emptyset,s}$, induces a surjection in cohomology

    \[
\frac{\QQ[x_1,\dots,x_n;u_1,\dots,u_s]}{( (x_i-u_1)\dots(x_i-u_s) : 1 \leq i \leq n)} \cong H^*_U(Y_{n,\emptyset,s}) \twoheadrightarrow H^*_U(Y_{n,\lambda,s})
\]
    which is the identity map on Chern classes. As a consequence, we have $(x_i-u_1)\dots(x_i-u_s) = 0$ for all $i$ in $J_{n,\lambda,s}$.
\end{lemma}

\begin{proof}
    This follows from Lemma \ref{lem: rel-paving-equiv-surj}, as well as Theorem 3.21, Lemma 4.5 in \cite{GLW}. The second statement is a consequence of Lemma \ref{lem: proj-bundle-cohom}.
\end{proof}

This implies that the first family of relations holds in $H^*_U(Y_{n,\lambda,s})$. The following lemma is also a consequence of the surjection $H^*_U(Y_{n,\emptyset,s}) \twoheadrightarrow H^*_U(Y_{n,\lambda,s})$:
\begin{lemma}\label{lem: vars-generate}
    The ring $H^*_U(Y_{n,\lambda,s})$ is generated by the variables $x_1,\dots,x_n$ and $u_1,\dots,u_s$ where
    \begin{itemize}
        \item $x_i = c_1^U(\ell_i)$ is the equivariant first Chern class of the $i$th tautological bundle $E_i/E_{i-1}$ over $Y_{n,\lambda,s}$
        \item $u_i$ is the first Chern class of the line bundle $EU \times_U \CC_i$ over $BU \cong EU \times_U \{\text{pt}\}$, and $\CC_i$ is the one dimensional representation of $U$ with action given by $(u_1,\dots,u_s)\cdot v = u_iv$.
    \end{itemize}
\end{lemma}

\subsection{An action of $\symm_n$ on $H^*_U(Y_{n,\lambda,s})$}

In this section, we construct an action of $\symm_n$ on
$H^*_U(Y_{n,\lambda,s})$ lifting the usual $\symm_n$-action on
$H^*(Y_{n,\lambda,s})\cong R_{n,\lambda,s}$. We will descend the evident action on the ambient
partial flag variety through the kernel of the restriction map.

Set
\[
S:=H_U^*(pt)=\QQ[u_1,\dots,u_s],
\qquad
A:=H_U^*(\Fl_{(1^n)}(\CC^K)).
\]
By the equivariant Borel presentation for partial flag varieties,
\[
A\cong
\left(
\frac{\QQ[\xxx_K,\uuu_s]}{(e_i(\xxx_K)-e_i(u_{\xi(1)},\dots,u_{\xi(K)}):1\leq i\leq K)}
\right)^{\symm_1\times\cdots\times \symm_1\times \symm_{K-n}},
\]
so permutation of the variables $x_1,\dots,x_n$ defines an action of $\symm_n$
on $A$ fixing $u_1,\dots,u_s$.

Let
\[
q^*:A\to H_U^*(Y_{n,\lambda,s})
\]
be the map induced by the inclusion
$q:Y_{n,\lambda,s}\hookrightarrow \Fl_{(1^n)}(\CC^K)$. By
Lemma~\ref{lem: vars-generate}, the ring $H_U^*(Y_{n,\lambda,s})$ is generated
by the restrictions of $x_1,\dots,x_n$ together with $u_1,\dots,u_s$, so $q^*$
is surjective.

\begin{proposition}\label{prop:Sn-action-equiv}
The kernel $\ker(q^*)$ is stable under the $\symm_n$-action on $A$. Hence the
$\symm_n$-action on $A$ descends to an action on $H_U^*(Y_{n,\lambda,s})$ with
\[
\sigma\cdot x_i=x_{\sigma(i)},\qquad \sigma\cdot u_j=u_j.
\]
Moreover, the map $q^*$ is $\symm_n$-equivariant.
\end{proposition}

\begin{proof}
Set
\[
S := H_U^*(pt)=\QQ[u_1,\dots,u_s],
\qquad
A := H_U^*(\Fl_{(1^n)}(\CC^K)),
\qquad
B := H_U^*(Y_{n,\lambda,s}),
\]
and let
\[
q^*:A \twoheadrightarrow B
\]
be the map induced by the inclusion
\[
Y_{n,\lambda,s}\hookrightarrow \Fl_{(1^n)}(\CC^K).
\]
By Lemma~\ref{lem: vars-generate}, the map \(q^*\) is surjective. Let
\[
I:=\ker(q^*),
\]
so that \(B\cong A/I\). Also set
\[
S_+ := \bigoplus_{d>0} S^{2d} = (u_1,\dots,u_s).
\]

Since both \(\Fl_{(1^n)}(\CC^K)\) and \(Y_{n,\lambda,s}\) admit affine pavings,
they are equivariantly formal by Lemma~\ref{lem: aff-paving-induces-basis}.
Hence the forgetful maps induce natural isomorphisms
\[
A/S_+A \cong H^*(\Fl_{(1^n)}(\CC^K)),
\qquad
B/S_+B \cong H^*(Y_{n,\lambda,s}).
\]
Under these identifications, the map \(q^*:A\to B\) induces the ordinary
restriction map
\[
\overline q^*: A/S_+A \longrightarrow B/S_+B,
\]
whose kernel is stable under the \(\symm_n\)-action by
\cite[Section~5]{GLW}.

Fix \(\sigma\in\symm_n\), and define the ideal
\[
\sigma(I):=\{\sigma(f):f\in I\}.
\]
Consider the graded \(S\)-submodule
\[
M:=\frac{\sigma(I)+I}{I}\subset A/I\cong B.
\]
To show that the \(\symm_n\)-action descends to \(A/I\), it is enough to prove
that \(M=0\), since this is equivalent to \(\sigma(I)\subseteq I\).

We now reduce modulo \(S_+\). Since \(M\subset A/I\), we have
\[
M/S_+M
\cong
\frac{\sigma(I)+I+S_+A}{I+S_+A}.
\]
Let
\[
\pi:A\twoheadrightarrow A/S_+A
\]
be the quotient map. Then
\[
\ker(\overline q^*) = \frac{I+S_+A}{S_+A} = \pi(I).
\]
Because the \(\symm_n\)-action fixes \(u_1,\dots,u_s\), it preserves \(S_+A\),
and therefore
\[
\pi(\sigma(I))=\sigma(\pi(I))=\sigma(\ker(\overline q^*)).
\]
It follows that
\[
M/S_+M
\cong
\frac{\sigma(\ker(\overline q^*))+\ker(\overline q^*)}
     {\ker(\overline q^*)}
=0,
\]
since \(\ker(\overline q^*)\) is \(\symm_n\)-stable.

Now \(B\) is a finitely generated graded \(S\)-module, so
Lemma~\ref{lem: graded-nakayama} (graded Nakayama) implies that \(M=0\). Hence
\[
\sigma(I)\subseteq I.
\]
Applying the same argument to \(\sigma^{-1}\) yields
\[
\sigma(I)=I.
\]

Therefore the \(\symm_n\)-action on \(A\) descends to an action on
\[
A/I \cong H_U^*(Y_{n,\lambda,s}).
\]
By construction, this descended action satisfies
\[
\sigma\cdot x_i=x_{\sigma(i)},
\qquad
\sigma\cdot u_j=u_j,
\]
and the map \(q^*\) is \(\symm_n\)-equivariant.
\end{proof}

\subsection{Family 2 and Schubert variety intersections}

We outline our strategy for proving the second family of relations, following the argument of Abe-Horiguchi \cite{AH}. First, we construct a map $p: Y_{n,\lambda,s} \to \Gr(m,K)$, which factors through $\Fl_{(1^n)}(\CC^K)$. We then show the image $p(Y_{n,\lambda,s})$ is contained a certain Schubert variety $X_{\mu_0}$. The polynomials $e_d(x_1,\dots,x_m| u_{\phi_\lambda(1)},\dots,u_{\phi_\lambda(n)})$ will turn out to be images of certain Schubert classes $X_{m,d}$ under the map $p^*$. A classical result shows the cells $X_{\mu_0}$ and $X_{\mu_{m,d}}$ do not meet, and so $e_d(x_1,\dots,x_m| u_{\phi_\lambda(1)},\dots,u_{\phi_\lambda(n)}) = 0$ in $H^*_U(Y_{n,\lambda,s})$. Applying the $\symm_n$-action from the previous section shows the vanishing of $e_d(x_{i_1},\dots,x_{i_m}| u_{\phi_\lambda(1)},\dots,u_{\phi_\lambda(n)})$ for arbitrary indices $1 \leq i_1 < \dots < i_m \leq n$.

We first recount a few important propositions and lemmata.

\begin{proposition}\label{prop: pullback-schub}
    Let $p: \Fl_K(\CC^K) \to \Gr(m,K)$ be given by $p: V_\bullet \mapsto V_m$. If $\mu$ is a partition that fits in a $m \times (K-m)$ rectangle, then the pullback of the Schubert class $X_\mu(F_\bullet)$ with respect to the standard flag $F_\bullet$ is given by

    \begin{equation}\label{eq: pullback-schub}
        p^*[X_\mu(F_\bullet)] = s_\mu(-\overline{x}_1,\dots,-\overline{x}_m|-t_K,\dots,-t_1,0,\dots)
    \end{equation}
\end{proposition}

\begin{proposition}\label{prop: parabolic-inclusion}\cite{AF}
    Let $\alpha \vDash n$ be a composition of $n$. Let $\pi: \Fl_n(\CC^n) \to \Fl_\alpha(\CC^n)$ be the projection map to the partial flag variety indexed by $\alpha$. We have that

    \begin{equation}\label{eq: parabolic-inclusion}
        H^*_T(\Fl_\alpha(\CC^n)) = H^*_T(\Fl_n(\CC^n))^{\symm_\alpha}
    \end{equation}
    the ring of invariants of $H^*_T(\Fl_n(\CC^n))$ under $\symm_\alpha := \symm_{\alpha_1} \times \dots \times \symm_{\alpha_\ell}$, where $\symm_\alpha$ acts on the $\xxx$-variables.
\end{proposition}

\begin{lemma}\label{lem: pullback-fl1nck-schur}
    Let $p_m:\Fl_{(1^n)}(\CC^K)\to \Gr(m,K)$ be defined by
    $p_m(V_\bullet)=V_m$ for $m\leq n$. Then
    \[
    p_m^*[X_\mu(F_\bullet)]
    =s_\mu(-\overline{x}_1,\dots,-\overline{x}_m\mid -t_K,\dots,-t_1,0,\dots).
    \]
\end{lemma}

\begin{proof}
Let $p_K:\Fl(\CC^K)\to \Gr(m,K)$ be defined by $p_K(V_\bullet)=V_m$. By
Proposition~\ref{prop: pullback-schub},
\[
p_K^*[X_\mu(F_\bullet)]
=s_\mu(-\overline{x}_1,\dots,-\overline{x}_m\mid -t_K,\dots,-t_1,0,\dots)
\in H_T^*(\Fl(\CC^K)).
\]
This class depends only on $\overline{x}_1,\dots,\overline{x}_m$, and hence is
invariant under the action of
$\symm_{1^n}\times \symm_{K-n}=\symm_1\times\cdots\times\symm_1\times\symm_{K-n}$
on the $\overline{x}$-variables. Therefore, by
Proposition~\ref{prop: parabolic-inclusion}, it lies in the invariant subring
\[
H_T^*(\Fl(\CC^K))^{\symm_{1^n}\times\symm_{K-n}}
\cong H_T^*(\Fl_{(1^n)}(\CC^K)),
\]
and this class is exactly $p_m^*[X_\mu(F_\bullet)]$.
\end{proof}

The following lemma is a modification of a proposition by Tanisaki \cite{T}.

\begin{lemma}\label{lem: ynls-inclusion-schub}
    Let $p:Y_{n,\lambda,s}\to \Gr(m,K)$ be given by $p(V_\bullet)=V_m$, and let
    $U_\bullet$ be a complete flag refining
    \[
    \cdots\subset N_\Lambda^2\CC^K\subset N_\Lambda\CC^K\subset \CC^K.
    \]
    Then
    \[
        p(Y_{n,\lambda,s})\subset X_{\mu_0}(U_\bullet),
    \]
    where
    \[
    \mu_0=(\underbrace{K-m,\dots,K-m}_{p_m^n(\lambda)},0,\dots,0).
    \]
\end{lemma}

\begin{proof}
    Let $V_\bullet\in Y_{n,\lambda,s}$. Using the shifted description of
    $Y_{n,\lambda,s}$, we have
    \[
    N_\Lambda^{n-k}\CC^K\subset V_n,
    \qquad
    N_\Lambda V_i\subset V_{i-1}\ \, (1\leq i\leq n).
    \]
    Hence
    \[
    N_\Lambda^{2n-k-m}\CC^K
    \subset N_\Lambda^{n-m}V_n
    \subset V_m.
    \]
    By construction of $N_\Lambda$, each application of $N_\Lambda$ removes one
    column of the canonical filling $P$. Therefore
    $N_\Lambda^{2n-k-m}\CC^K$ is spanned by the basis vectors corresponding to
    the boxes of $[\lambda]$ weakly to the right of the $(n-m)$th column, and so
    \[
    \dim N_\Lambda^{2n-k-m}\CC^K=p_m^n(\lambda).
    \]
    Since $U_\bullet$ refines the filtration by images of powers of $N_\Lambda$,
    we may choose it so that
    \[
    U_{p_m^n(\lambda)}=N_\Lambda^{2n-k-m}\CC^K.
    \]
    Thus
    \[
    U_{p_m^n(\lambda)}\subset V_m.
    \]
    It follows that for $1\leq i\leq p_m^n(\lambda)$,
    \[
    \dim(V_m\cap U_i)=i.
    \]
    For the remaining indices $p_m^n(\lambda)\leq i\leq m$, we use the dimension
    estimate
    \[
    \dim(V_m\cap U_{K-m+i})
    \geq \dim V_m + \dim U_{K-m+i} - K
    = m + (K-m+i)-K = i.
    \]
    These are exactly the Schubert conditions defining
    $X_{\mu_0}(U_\bullet)$, so $V_m\in X_{\mu_0}(U_\bullet)$.
\end{proof}
This implies that $p$ restricts to a map
$\widehat{p}:Y_{n,\lambda,s}\to X_{\mu_0}(U_\bullet)$, and we obtain a
commutative diagram
\begin{center}
    \begin{tikzcd}
    H_T^*(\Fl_{(1^n)}(\CC^K)) \arrow[d, "{p_{n,\lambda,s}^*}"] &  & {H_T^*(\Gr(m,K))} \arrow[d, "i^*"] \arrow[ll, "p_m^*"] \\
    {H_U^*(Y_{n,\lambda,s})}                                   &  & H_T^*(X_{\mu_0}(U_\bullet)) \arrow[ll, "\widehat{p}^*"]
    \end{tikzcd}
\end{center}
where $i$ is the inclusion $X_{\mu_0}(U_\bullet)\hookrightarrow \Gr(m,K)$.

We briefly recall a standard fact.

\begin{lemma}\label{lem:eq-class-supported-vanishes}
    Let $X$ be a smooth $T$-variety, let $Y \subset X$ be a closed
    $T$-stable subvariety, and let $i:Z\hookrightarrow X$ be the inclusion
    of a closed $T$-stable subvariety. Suppose the class
    \[
        \alpha \in H_T^*(X)
    \]
    is supported on $Y$, i.e. lies in the image of the natural map
    \[
        H_T^*(X,X\setminus Y)\to H_T^*(X).
    \]
    If $Z\cap Y=\varnothing$, then
    \[
        i^*(\alpha)=0 \in H_T^*(Z).
    \]
\end{lemma}

\begin{proof}
    Since $\alpha$ comes from $H_T^*(X,X\setminus Y)$, its pullback
    $i^*(\alpha)$ comes from
    \[
        H_T^*(Z,Z\setminus (Z\cap Y)).
    \]
    If $Z\cap Y=\varnothing$, then this relative group is
    $H_T^*(Z,Z)=0$, so $i^*(\alpha)=0$.
\end{proof}

\begin{lemma}\label{lem: musd-vanishing}
    Let
    \[
    \mu_{m,d}:=(\underbrace{1,\dots,1}_{d},0,\dots,0),
    \]
    where there are $d$ copies of $1$ and $m-d$ copies of $0$. If
    $d>m-p_m^n(\lambda)$, then
    \[
    i^*[X_{\mu_{m,d}}(w_U\widetilde{F}_\bullet)]_T=0.
    \]
    Consequently,
    \[
    p^*[X_{\mu_{m,d}}(w_U\widetilde{F}_\bullet)]_T=0.
    \]
\end{lemma}

\begin{proof}
    By Lemma~\ref{lem: grass-intersect}, if $d>m-p_m^n(\lambda)$, then
    $\mu_{m,d}\not\subset \mu_0^*$, and hence
    \[
    X_{\mu_0}(U_\bullet)\cap X_{\mu_{m,d}}(w_U\widetilde{F}_\bullet)=\varnothing.
    \]
    Therefore by Lemma \ref{lem:eq-class-supported-vanishes} the equivariant fundamental class of
    $X_{\mu_{m,d}}(w_U\widetilde{F}_\bullet)$ restricts trivially to
    $X_{\mu_0}(U_\bullet)$, i.e.
    \[
    i^*[X_{\mu_{m,d}}(w_U\widetilde{F}_\bullet)]_T=0.
    \]
    Pulling back along $\widehat{p}$ gives the second statement.
\end{proof}

\begin{lemma}\label{lem: eq-class-musd}\cite{AH}
    Let $\mu_{m,d}$ be as above. Then
    \begin{equation}\label{eq: schur-musd}
        s_{\mu_{m,d}}(x_1,\dots,x_m\mid a_1,a_2,\dots)
        =\sum_{r=0}^{d} (-1)^{d-r}
        e_r(x_1,\dots,x_m)h_{d-r}(a_1,\dots,a_{m+1-d}).
    \end{equation}
\end{lemma}

\begin{proposition}\label{prop: family2}
    For $d>m-p_m^n(\lambda)$, we have
    \[
    p^*[X_{\mu_{m,d}}(w_U\widetilde{F}_\bullet)]_T
    =(-1)^d e_d(x_1,\dots,x_m\mid u_{\phi_\lambda(1)},\dots,u_{\phi_\lambda(n)}).
    \]
    Consequently,
    \[
    e_d(x_{i_1},\dots,x_{i_m}\mid u_{\phi_\lambda(1)},\dots,u_{\phi_\lambda(n)})=0
    \]
    in $H_U^*(Y_{n,\lambda,s})$ for every choice of indices
    $1\leq i_1<\cdots<i_m\leq n$.
\end{proposition}

\begin{proof}
    By choice of $w_U$, the flag $U_\bullet=w_UF_\bullet$ refines the filtration
    by powers of $N_\Lambda$. By Lemmas~\ref{lem: musd-vanishing},
    \ref{lem: pullback-fl1nck-schur}, and \ref{lem: eq-class-musd}, we compute
    \begin{align*}
        0
        &=p^*[X_{\mu_{m,d}}(w_U\widetilde{F}_\bullet)]_T \\
        &=p_{n,\lambda,s}^*\,p_m^*[X_{\mu_{m,d}}(w_U\widetilde{F}_\bullet)]_T \\
        &=p_{n,\lambda,s}^*\,s_{\mu_{m,d}}(-\overline{x}_1,\dots,-\overline{x}_m
        \mid -t_{w_U(1)},\dots,-t_{w_U(K)},0,\dots) \\
        &=(-1)^d\,p_{n,\lambda,s}^*\,e_d(\overline{x}_1,\dots,\overline{x}_m
        \mid t_{w_U(1)},\dots,t_{w_U(m+1-d)}) \\
        &=(-1)^d e_d(x_1,\dots,x_m\mid u_{\xi(w_U(1))},\dots,u_{\xi(w_U(m+1-d))}) \\
        &=(-1)^d e_d(x_1,\dots,x_m\mid u_{\phi_\lambda(1)},\dots,u_{\phi_\lambda(m+1-d)}).
    \end{align*}
    Since the factorial elementary symmetric polynomial only depends on the
    first $m+1-d$ auxiliary parameters, this is exactly
    \[
    (-1)^d e_d(x_1,\dots,x_m\mid u_{\phi_\lambda(1)},\dots,u_{\phi_\lambda(n)}).
    \]

    This proves the relation for the first $m$ variables. For an arbitrary
    $m$-subset $\{i_1<\cdots<i_m\}\subset [n]$, choose
    $\sigma\in \symm_n$ with $\sigma(\{1,\dots,m\})=\{i_1,\dots,i_m\}$. Since
    the $\symm_n$-action of Proposition~\ref{prop:Sn-action-equiv} fixes the
    $u$-variables and permutes the $x$-variables, applying $\sigma$ to the above
    vanishing relation gives
    \[
    e_d(x_{i_1},\dots,x_{i_m}\mid u_{\phi_\lambda(1)},\dots,u_{\phi_\lambda(n)})=0.
    \]
\end{proof}

\begin{corollary}[Theorem~\ref{thm: main}]
    We have the ring isomorphism
    \begin{equation}
        H^*_U(Y_{n,\lambda,s}) \cong \QQ[\xxx_n,\uuu_s]/J_{n,\lambda,s}
    \end{equation}
\end{corollary}

\begin{proof}
    By lemma \ref{lem: proj-bundle-surjection} and proposition \ref{prop: family2}, we have that the generators of the ideal $J_{n,\lambda,s}$ vanish in $H_U^*(Y_{n,\lambda,s})$, so the map $p: \QQ[\xxx_n,\uuu_s]/J_{n,\lambda,s} \to H^*_U(Y_{n,\lambda,s})$ is well defined. By lemma \ref{lem: vars-generate}, the map is surjective. By proposition \ref{prop: module-dim} and corollary \ref{cor: moduledim2}, we have that $H^*_U(Y_{n,\lambda,s})$ and $\QQ[\xxx_n,\uuu_s]/J_{n,\lambda,s}$ are free $\QQ[\uuu_s]$-modules of the same rank, and so $p$ is an isomorphism and the claim is proved.
\end{proof}

\section{Some Remarks on Localization}

We record here a localization statement for the $U$-action on
$Y_{n,\lambda,s}$. Since $Y_{n,\lambda,s}$ admits a $U$-stable affine paving,
its equivariant cohomology is a free $H_U^*(pt)$-module. Therefore the
Chang-Skjelbred sequence \cite{CS} gives an injection
\[
H_U^*(Y_{n,\lambda,s}) \hookrightarrow H_U^*(Y_{n,\lambda,s}^U).
\]
Thus the ring $H_U^*(Y_{n,\lambda,s})$ may be viewed as a subring of the
equivariant cohomology of the fixed locus.

At present, however, we do not have a satisfactory combinatorial description of
the image of this restriction map. Unlike the classical GKM setting, the
$U$-fixed locus is not expected to be finite: the torus $U$ acts with repeated
characters along the rows of $\Lambda$, so localization does not reduce to a
finite direct sum of copies of $\QQ[\uuu_s]$. A more precise description would
require understanding the geometry of the fixed components
$Y_{n,\lambda,s}^U$ together with the closures of the one-dimensional
$U$-orbits. We leave such a non-isolated GKM description for future work.

There is also a useful spectrum-theoretic viewpoint due to
Goresky--MacPherson \cite{GM}. Roughly speaking, for many equivariantly formal torus
actions, the affine scheme $\Spec H_T^*(X;\CC)$ is supported on an arrangement
of linear subspaces of $H_T^2(X;\CC)$, with one component associated to each
connected component of the fixed locus. When $H^*(X;\CC)$ is generated in
degrees $0$ and $2$, the fiber of this arrangement over $0$ recovers the
ordinary cohomology ring, and generic fibers often control associated-graded
descriptions of $H^*(X;\CC)$. In our setting, this applies as follows:

\begin{remark}[Spectrum viewpoint]
By Theorem~\ref{thm: main} and Proposition~\ref{prop: module-dim}, we have
\[
\mathrm{Spec}\, H_U^*(Y_{n,\lambda,s}) \cong X_{n,\lambda,s}\subset \mathbb A^{n+s}.
\]
Thus the torus-equivariant cohomology ring of the $\Delta$-Springer fiber is
realized as the coordinate ring of the universal arrangement
$X_{n,\lambda,s}$.

Since $H_U^*(Y_{n,\lambda,s})$ is a free $\QQ[\uuu_s]$-module, the projection
\[
X_{n,\lambda,s}\longrightarrow \mathbb A^s
\]
to the parameter space is finite and flat. The special fiber over
$\uuu_s=\mathbf 0$ recovers the ordinary cohomology ring
$H^*(Y_{n,\lambda,s})$, while fibers over tuples of distinct parameters recover
the finite loci $X_{n,\lambda,s}(\aalpha_s)$ appearing in Griffin's
orbit-harmonics construction. In this sense, our presentation may be viewed as a
concrete instance of the spectrum-theoretic viewpoint of Goresky and
MacPherson.
\end{remark}

\section{Future directions}

The results of this paper suggest several natural problems.

\begin{question}
Describe the fixed locus $Y_{n,\lambda,s}^U$ and the corresponding
one-dimensional $U$-orbit closures. Is there a natural combinatorial model for
these fixed components in terms of the tableau $P$ or the partition data
$(n,\lambda,s)$?
\end{question}

\begin{question}
Does the $U$-action on $Y_{n,\lambda,s}$ admit a non-isolated GKM description?
More precisely, can one recover $H_U^*(Y_{n,\lambda,s})$ from the equivariant
cohomology of the connected components of $Y_{n,\lambda,s}^U$ together with
compatibility conditions coming from codimension-one subtori?
\end{question}

\begin{question}
Theorem~\ref{thm: main} identifies $\Spec H_U^*(Y_{n,\lambda,s})$ with the
universal arrangement $X_{n,\lambda,s}$. Can one make the spectrum-theoretic
connection with orbit harmonics more precise by comparing the generic fibers of
\[
X_{n,\lambda,s}\to \mathbb A^s
\]
with the filtrations appearing in the sense of Goresky-MacPherson?
\end{question}

\begin{question}
Our construction gives an action of $\symm_n$ on $H_U^*(Y_{n,\lambda,s})$.
Can this action be interpreted directly on the arrangement
$X_{n,\lambda,s}\to \mathbb A^s$? More generally, is there a natural
characterization of the automorphisms of $\Spec H_U^*(Y_{n,\lambda,s})$ that
preserve the projection to $\Spec H_U^*(pt)$?
\end{question}

\bibliographystyle{plain}
\bibliography{refs}

\end{document}